\documentclass[10pt,a4paper]{amsart}

\usepackage[utf8]{inputenc}
\usepackage{t1enc}
\usepackage{microtype}

\usepackage{amssymb,amsfonts}
\usepackage{mathtools}

\usepackage[margin=2.7cm,marginpar=2cm]{geometry}

\usepackage[unicode]{hyperref}

\usepackage[capitalize,nameinlink]{cleveref}
\numberwithin{equation}{section}

\def\equationautorefname~#1\null{Equation~(#1)\null}

\let\eps\varepsilon
\usepackage{thmtools}

\declaretheorem[
style=plain,
name=Theorem,
numbered=yes,
refname={Theorem,Theorems},
Refname={Theorem,Theorems}
]{Thm}
\declaretheorem[
style=plain,
name=Proposition,
numberlike=Thm,
refname={Proposition,Propositions},
Refname={Proposition,Propositions}
]{Prop}

\declaretheorem[
style=definition,
name=Example,
numberlike=Thm,
refname={Example,Examples},
Refname={Example,Examples}
]{Eg}
\declaretheorem[
style=plain,
name=Corollary,
numberlike=Thm,
refname={Corollary,Corollaries},
Refname={Corollary,Corollaries}
]{Cor}
\declaretheorem[
style=plain,
name=Lemma,
numberlike=Thm,
refname={Lemma,Lemmas},
Refname={Lemma,Lemmas }
]{Lem}

\newcommand{\dd}{\mathrm{d}}

\usepackage{nicefrac}
\newcommand{\half}{{\nicefrac{1\mskip-2mu}{\mskip-2mu2}}}

\usepackage[normalem]{ulem}
\renewcommand{\vec}[1]{\uline{\boldsymbol{#1}}}

\newcommand{\poch}[2]{\{#1\}_{#2}}
\DeclareMathOperator{\wwt}{wt}
\DeclareMathOperator{\ddp}{dp}
\DeclareMathOperator{\hht}{ht}
\newcommand{\Z}{\mathbb{Z}}

\usepackage{shuffle}
\DeclareMathOperator{\reg}{reg}
\DeclareMathOperator{\Li}{Li}

\newmuskip\pFqmuskip
\newcommand*\pFq[6][8]{%
	\begingroup 
	\pFqmuskip=#1mu\relax
	\mathchardef\normalcomma=\mathcode`,
	\mathcode`\,=\string"8000
	\begingroup\lccode`\~=`\,
	\lowercase{\endgroup\let~}\pFqcomma
	{}_{#2}F_{#3}{\left[\genfrac..{0pt}{}{#4}{#5};#6\right]}%
	\endgroup
}
\newcommand{\pFqcomma}{{\normalcomma}\mskip\pFqmuskip}

\newcommand{\indsh}{\mathbin{\widetilde{\shuffle}}}

\usepackage{bbm}
\newcommand{\one}{\mathbbm{1}}

\begin{document}
	
	\title{Multiple zeta values ending with a fixed string}
	\date{June 8, 2026}
	
\author{Steven Charlton}
\address{}
\email{mail@stevencharlton.net}

	\subjclass[2020]{}
	

	\begin{abstract} 
	We give a generating series expression for the sum of all multiple zeta values of a fixed weight, depth, and height, which end with a given string \( \vec{\ell} = (\ell_1,\ldots,\ell_r) \); this builds upon the proof of the Ohno-Zagier Theorem.  In particular, the sum of all multiple zeta values of fixed weight, depth and ending with \( \vec{\ell} \) has bounded depth \( \leq \ell_1 + \cdots + \ell_r \).  We give some applications to evaluations of interpolated multiple zeta values, and to the generating series of double zeta values.
	\end{abstract}

	\maketitle

	\section{Introduction}
	
	For any index \( \vec{k} = (k_1,\ldots,k_d) \in \mathbb{Z}_{>0}^d \), 
	\begin{itemize}
		\item the \emph{weight} is \(  \wwt(\vec{k}) = k_1 + \cdots + k_d \),  the sum of all entries;
		\item the \emph{depth} is \( \ddp(\vec{k}) = d \), the number of entries; and
		\item the \emph{height} is \( \hht(\vec{k}) =  \#\{ i \mid k_i > 1 \} \), the number of entries \( >1 \).
	\end{itemize}
	We call a index \( \vec{k} = (k_1, \ldots, k_d) \) \emph{admissible} when \( k_d \geq 2 \).  Write
	\begin{align*}
		I(w,d,h) & {}= \{ \vec{k} = (k_1,\ldots,k_d) \in \Z_{>0}^d \mid \wwt(\vec{k}) = w, \hht(\vec{k}) = h \} \,, \text{ and } \\
		I_0(w,d,h) & {}= \{ \vec{k} = (k_1,\ldots,k_d) \in \Z_{>0}^d \mid \wwt(\vec{k}) = w, \hht(\vec{k}) = h, \text{$k_d \geq 2$}\} 
	\end{align*}
	for the set of all indices of given weight \( w \), depth \( d \) and height \( h \), and for the subset of admissible ones, respectively.  Write 
	\(
	 I(w,d) = \bigcup\nolimits_{h=1}^\infty I(w,d,h) \), and \(  I_0(w,d) = \bigcup\nolimits_{h=1}^\infty I_0(w,d,h) 
	\)
	for the corresponding sets without restrictions on the height \( h \).  For an admissible index \(\vec{k}\), define the \emph{multiple zeta value} (MZV) as
	\begin{equation}\label{eqn:def:mzv}
		\zeta(\vec{k}) = \zeta(k_1, k_2, \ldots, k_d) \coloneqq \sum_{0 < n_1 < n_2 < \cdots < n_d} \frac{1}{n_1^{k_1} n_2^{k_2} \cdots n_d^{k_d}} \,.
	\end{equation}\medskip
	
	Ohno-Zagier \cite{OhnoZagier01} studied the generating series for sums of MZV's of fixed weight, depth and height.
	\begin{Thm}[Ohno-Zagier, \cite{OhnoZagier01}]
		The generating series for the sum of all MZV's of fixed weight \( w \), depth \( d \), and height \( h \) is given by
		\[
			\sum_{w, d, h \geq 1 } \biggl( \, \sum\nolimits_{\substack{\vec{k} \in I_0(w, d, h)}} \zeta(\vec{k}) \! \biggr) x^{w - d - h} y^{d-h} z^{h} = \frac{1}{x y - z} \Biggl( 1 - \exp\Biggl( \sum_{n=2}^\infty \frac{\zeta(n)}{n} S_n(x,y,z) \Biggr) \! \Biggr) \,,
		\]
		where \( S_n(x,y,z) \in \mathbb{Z}[x,y,z] \) is the family of polynomials defined by
		\[
			S_n(x,y,z) = x^n + y^n - \alpha^n - \beta^n \,, \quad \alpha,\beta = \frac{x + y \pm \sqrt{(x+y)^2 - 4 z}}{2} \,.
		\]
	\end{Thm}
	In particular, the sum of all MZV's of fixed weight \( w \), depth \( d \) and height \( h \) is a polynomial in Riemann zeta values.  The main result of this note is a version where \( \vec{k} \) is required to end in a given string \( \vec{\ell} \).  In particular, if we weight by a binomial involving the height, the sums have restricted depth.
	\begin{Thm}[\autoref{cor:wtsum:ell}]\label{thm:weightedsum} 
		Fix the weight \( w \), depth \( d \) and a parameter \( \eta \).  Fix a admissible index \( \vec{\ell} = (\ell_1,\ldots,\ell_r) \) with \( \ell_r > 1 \).  Then the weighted sum,
		\[
			\sum\nolimits_{\substack{\vec{k} \in I_0(w,d) \\ \text{$\vec{k}$ end in $\vec{\ell}$}}} \binom{\hht(\vec{k})-1}{\eta} \zeta(\vec{k}) 
		\]
		has depth \( \leq \eta + \wwt(\vec{\ell}) \), independent of \( w \) and \( d \).
	\end{Thm}
	We will state and prove a more precise result in \autoref{thm:phiell:gs}, which gives an expression for these sums in terms of shuffle-regularised MZV's of the required depth after duality. 	We will recall the details of shuffle regularisation in \autoref{sec:reg}; for simplicity in the introduction, we can take as a definition
	\[
	\reg_{\shuffle,0} \zeta_{k_0}(k_1,\ldots,k_d) \coloneqq (-1)^{k_0} \sum_{\substack{i_1 + \cdots + i_d = k_0 \\ i_1,\ldots,j_d\geq 0}} \tbinom{k_1 + i_1 - 1}{i_1} \cdots \tbinom{k_d + i_d - 1}{i_d} \zeta(k_1 + i_1, \ldots, k_d + i_d) \,,
	\]
	for \( k_d \geq 2 \).	As an example of \autoref{thm:weightedsum}, we have the following
	\begin{Eg}
		In the case \( \vec{\ell} = (2,2) \) and \( \eta = 0 \), the following identity holds
	\begin{align*}
		& \sum\nolimits_{\substack{\vec{k} \in I_0(w,d) \\ \text{$\vec{k}$ ends in $(2,2)$}}} \zeta(\vec{k}) = \\
		& (-1)^{w+d} \reg_{\shuffle,0} \Bigg\{ \, \begin{aligned}[t]
		& \sum_{\substack{i+j+k\\=w-4}} \tbinom{i}{d-2} \zeta_k(j+2,i+2)
		+\sum_{\substack{i+j+k\\=w-5}} \tbinom{i}{d-2} \zeta_k(1,j+2,i+2)
		\\[-1ex]
		&+\sum_{\substack{i+j+k\\=w-5}} \tbinom{i}{d-2} \zeta_k(j+2,1,i+2)
		+\sum_{\substack{i+j+k\\=w-6}} \tbinom{i}{d-2} \zeta_k(1,j+2,1,i+2) \Bigg\} \,,
		\end{aligned}		
	\end{align*}
	which invokes only MZV's of depth \( \leq \wwt(\vec{\ell}) = 4 \). 
	\end{Eg}
	
	As an application of \autoref{thm:weightedsum}, we will obtain an evaluation for certain interpolated multiple zeta values.   Recall \emph{multiple zeta star values} (MZSV's) are defined by replacing the strict inequalities between summation variables in \autoref{eqn:def:mzv} with non-strict ones, namely  
		\begin{equation}\label{eqn:def:mzsv}
	\zeta^\star(\vec{k}) = \zeta^\star(k_1, k_2, \ldots, k_d) \coloneqq \sum_{0 < n_1 \leq n_2 \leq \cdots \leq n_d}  \frac{1}{n_1^{k_1} n_2^{k_2} \cdots n_d^{k_d}} \,.
	\end{equation}	
	The \emph{interpolated} multiple zeta values were introduced by Yamamoto \cite{YamamotoInterpolation13} to naturally connect the usual MZV's (at $r=0$) with the MZSV's (at $r=1$).  The interpolated MZV's are defined by
	\begin{align*}
	\zeta^r(k_1,\ldots,k_d) & 
	\coloneqq \sum_{0 < n_1 \leq n_2 \leq \cdots \leq n_d} \frac{r^{\#\{ i \mid n_i = n_{i+1} \}}}{n_1^{k_1} n_2^{k_2} \cdots n_d^{k_d}}\
	 = \sum_{\circ_i = \text{`$+$' or `,'}} r^{\#\{ i \mid \circ_i = {+}\}}\zeta(k_1 \circ_1 k_2 \circ_2 \cdots \circ_{k_{d-1}} k_d)
	\end{align*}
	From \autoref{thm:weightedsum}, we can then derive the following evaluation.
	\begin{Prop}[\autoref{prop:zr:2112}]
		The following evaluation holds for all \( n \geq 0 \), 
		\[
			\zeta^r(2, \{1\}^n, 2) = \frac{1}{(1-r)r^2} \sum_{\substack{i+j = n+4 \\ i \geq 1, j \geq 2}} r^j \zeta(i, j) - \frac{r^{n+4}}{(1-r)r^2} \zeta(n+4) \,.
		\]
	\end{Prop}
	When \( r = \half \), this simplifies to a polynomial in single zeta values using a weighted sum formula.%
	\begin{Prop}[\autoref{prop:zh:2112}]
		The following evaluation hold for all \( n \geq 0 \),
		\begin{align*}
		\zeta^\half(2,\{1\}^n,2) = \sum_{\substack{i+j=n+4 \\ i,j \geq 2}} {-}2^{1-i-j}(2 - 2^i)(2 - 2^j) \zeta(i)\zeta(j) + (2^{-1-n} + 2n + 2) \zeta(n+4) \,.
		\end{align*}
	\end{Prop}
	  Finally, we also derive a \( {}_4F_3 \) hypergeometric expression for the generating series of double zeta values in \autoref{prop:dzv:hyper}.  This could be used to study MZV evaluations via double zeta values, using automatic summation, creative telescoping and WZ methods as in \cite{AuCtel}.
	
	\subsection*{Acknowledgements} I am grateful to Max Planck Institute for Mathematics in
	Bonn for support and hospitality, during the earliest stages of this work.  I would like to thank Nobuo Sato and Wadim Zudilin for useful discussions on  hypergeometric series, which helped to further extend some results, and  Michael Hoffman for frequent discussions on the multiple zeta values and related topics.
	
	\section{Review of quasi-shuffle algebras, and regularisation}\label{sec:reg}

	Here we recall the details of the algebraic framework for describing the structure of multiple zeta values, and the two regularisation constructions for multiple zeta values of non-admissible indices.
	
	\subsection*{Stuffle product} Multiple zeta values satisfy a product structure from multiplying the series representation in \autoref{eqn:def:mzv}.  This is called the \emph{stuffle product} or \emph{harmonic product}.  For example
	\begin{align*}
	&\zeta(k_1) \zeta(\ell_1,\ell_2) \, = \, \sum_{0 < m_1} \sum_{0 < n_1 < n_2} \frac{1}{m_1^{k_1}} \cdot \frac{1}{n_1^{\ell_1} n_1^{\ell_2}} \\
	& = \bigg( \sum_{0 < m_1 < n_1 < n_2} + \sum_{0 <  n_1 < m_1 < n_2} + \sum_{ 0 < n_1 < n_2 < m_1} + \sum_{0 < n_1 = m_1 < n_2} + \sum_{0 < n_1 < m_1 = n_2} \bigg) \frac{1}{m_1^{k_1}} \cdot \frac{1}{n_1^{\ell_1} n_1^{\ell_2}}  \\[1ex]
	& = \zeta(k_1,\ell_1,\ell_2) + \zeta(\ell_1,k_1,\ell_2) + \zeta(\ell_1,\ell_2,k_1) + \zeta(k_1 + \ell_1, \ell_2) + \zeta(\ell_1, k_1 + \ell_2) \,.
	\end{align*}
	The product is computed by interleaving the summation indices in all possible ways, and explicitly allowing equality between summation indices from the different series.  \medskip
	
	\subsection*{Integral representation, and shuffle product}  Multiple zeta values can also be written as iterated integrals.  Namely
	\[
		\zeta(k_1,\ldots,k_d) = I(0; \underbrace{1,0,\ldots,0}_{k_1}, \underbrace{1,0,\ldots,0}_{k_2},\ldots,\underbrace{1,0,\ldots,0}_{k_d}; 1) \,,
	\]
	where
	\[
		I(a; \eps_1, \eps_2, \ldots, \eps_w; b) = \int_{a < t_1 < \cdots < t_w < b} \omega_{\eps_1}(t_1) \omega_{\eps_2}(t_2) \cdots \omega_{\eps_w}(t_w)  \,,
	\]
	with \( \eps_1(t) = \frac{\dd t}{1-t} \) and \( \eps_0(t) = \frac{\dd t}{t} \).  The integral representation makes clear that multiple zeta values satisfy a \emph{duality}: under the change of variables \( t_i \mapsto 1-t_i \), namely exchanging \( 0 \leftrightarrow 1 \) and reversing the string \( \eps_1,\ldots,\eps_w \), one obtains the integral representation of a different MZV which must have the same numerical value.  For example \( \zeta(1,2,4) = \zeta(1,1,2,3) \).
	
	Multiple zeta values inherit a second product structure from multiplication of the integral representation.  This is called the \emph{shuffle product}.   For example:
	\begin{align*}
		\zeta(2) \zeta(2) &= \int_{0 < t_1 < t_2 < 1} \frac{\dd t_1}{1-t_1} \frac{\dd t_2}{t_2} \int_{0 < s_1 < s_2 < 1} \frac{\dd s_1}{1-s_1} \frac{\dd s_2}{s_2} \\[1ex]
		&= 4 \zeta(1,3) + 2 \zeta(2,2) \,,
	\end{align*}
	by considering how \( 0 < t_1 < t_2 < 1 \) and \( 0 < s_1 < s_2 < 1 \) can interleave.  Here the equality \( s_i = t_j \) defines a measure 0 set, so does not contribute to the resulting integral. \medskip
	
	\subsection*{Algebraic framework, shuffle and stuffle product} We describe these two products algebraically using the setup of \cite{HoffmanAlgebra97} and \cite{IKZ06}.  Let \( \mathfrak{H} = \mathbb{Q}\langle x, y \rangle \) be the non-commutative polynomial algebra on two indeterminates, and set \( \mathfrak{H}^0 \coloneqq \mathbb{Q} + y \mathfrak{H} x \).  Let \( z_k \coloneqq y x^{k-1} \); note every element of \( \mathfrak{H}^0 \) can be expressed as a concatenation of suitable \( z_{k_i} \)'s, the last of which is not \( z_1 \).  Define the map
	\begin{align*}
		\zeta \colon \mathfrak{H}^0 & \to \mathbb{R} \\
		z_{k_1} z_{k_2} \cdots z_{k_d} &= \zeta(k_1,k_2,\ldots,k_n) \,.
	\end{align*}
	
	Write \( \mathfrak{H}^1 \coloneqq \mathbb{Q} + y \mathfrak{H}\), and denote by \( \mathbbm{1} \) the empty word.  For any words \( w, w_1, w_2 \in \mathfrak{H}^1 \), define the stuffle product via
	\begin{align*}
		&\left\{ \begin{aligned}
		\quad z_{k_1} w_1 \ast z_{k_2} w_2 &= z_{k_1} (w_1 \ast z_{k_2} w_2) + z_{k_2} (z_{k_1} w_1 \ast w_2) + z_{k_1+k_2} (w_1 \ast w_2) \,; \\
		w \ast \one &= \one \ast w = w\,,
		\end{aligned} \right. 
		\intertext{and extended by linearity.  Likewise for \( u_1, u_2 \in \{ x, y \} \), and words \( w, w_1, w_2 \in \mathfrak{H} \) define the shuffle product via}
		&\left\{ \begin{aligned}
		\quad u_1 w_1 \shuffle u_2 w_2 &= u_1 (w_1 \shuffle u_2 w_2) + u_2 (u_1 w_1 \shuffle w_2) \,. \\
		w \shuffle \one & = \one \shuffle w = w \,,
		\end{aligned} \right.
	\end{align*}
	and extended by linearity.  Write \( \mathfrak{H}^0_\bullet \coloneqq  (\mathfrak{H}^0, \bullet) \) for \( \bullet \in \{ \ast, \shuffle \} \), by \cite{HoffmanAlgebra97}, \cite{Reutenauer93} these are algebras.  For \( w_1, w_2 \in \mathfrak{H}^0 \), we have, by the discussion above, that
	\[
		\zeta(w_1 \ast w_2) = \zeta(w_1) \zeta(w_2) = \zeta(w_1 \shuffle w_2) \,,
	\]
	showing that \( \zeta \) is an algebra morphism from \( \mathfrak{H}^0_\ast \coloneqq (\mathfrak{H}^0, \ast) \), resp. \( \mathfrak{H}^0_\shuffle \coloneqq (\mathfrak{H}^0, \shuffle) \), to \( \mathbb{R} \).

	\subsection*{Stuffle and shuffle regularisation}
	
	Write \( \mathfrak{H}^1_\bullet \coloneqq  (\mathfrak{H}^1, \bullet) \) for \( \bullet \in \{ \ast, \shuffle \} \).  By \cite{HoffmanAlgebra97}, \cite{Reutenauer93} one has
	\[
		\mathfrak{H}_\bullet \cong \mathfrak{H}^0_\bullet[y] \,.
	\]
	One then defines \cite[Proposition 1]{IKZ06} morphisms
	\[
		\reg_{\ast,T} \zeta \colon \mathfrak{H}^1_\ast \to \mathbb{R}[T] \qquad \text{and} \qquad
		\reg_{\shuffle,T} \zeta \colon \mathfrak{H}^1 \to \mathbb{R}[T] \,,
	\]
	by extending the map \( \zeta \), and sending \( y \) to \( T \).  These are called the \emph{stuffle regularisation} and \emph{shuffle regularisation}, respectively.  Moreover, \( \mathfrak{H}_\shuffle \coloneqq (\mathfrak{H}, \shuffle) \cong \mathfrak{H}^1_\shuffle[x] \), so one can further define
	\[
		\reg_{\shuffle,T,S} \zeta \colon \mathfrak{H}_\shuffle \to \mathbb{R}[T,S]		
	\]
	by sending \( y \) to \( T \) and \( x \) to \( S \).  The duality is preserved by imposing \( T = S \), and we simply write \( \reg_{\shuffle,T} \colon \mathfrak{H}_\shuffle \to \mathbb{R}[T] \) in this case.
	
	One of the main result of \cite{IKZ06} is a comparison map to relate the shuffle and stuffle regularisations, which are in general different.  Write
	\[
		A(u) = e^{\gamma u} \Gamma(1+u) = \exp\Bigg(\sum_{n=2}^\infty \frac{(-1)^n}{n} \zeta(n) u^n\Bigg) \,,
	\]
	and define the \( \mathbb{R} \)-linear map \( \rho \colon \mathbb{R}[T] \to \mathbb{R}[T] \) by \[
		\rho(e^{Tu}) = A(u) e^{Tu} \,,
	\]
	Then \cite[Thm. 1]{IKZ06} establishes that for any \( w \in \mathfrak{H}^1 \),
	\( \reg_{\shuffle,T} \zeta(w) = \rho( \reg_{\ast,T} \zeta(w)) \,. \)  
	
	By abuse of notation we may write \( \reg_{\bullet,T} \zeta(k_1,\ldots,k_d) \), \( \bullet \in \{ \ast, \shuffle \} \) to mean \( \reg_{\bullet,T} \zeta(z_{k_1} \cdots z_{k_d}) \).  For any word \( w \) we write \( \reg_{\shuffle,T} \zeta_{k_0}(w) \) to denote \( \reg_{\shuffle,T} \zeta(x^{k_0} w) \) with \( k_0 \) initial \( x \)'s.  By abuse of notation we may also write \( \reg_{\shuffle,T} \zeta_{k_0}(k_1,\ldots,k_d) \) for this.  By induction one can show the following explicit formula (see I2 in \cite{BrownMTM12}), for computing the shuffle regularisation at \( T = 0 \),
	\[
	\reg_{\shuffle,0} \zeta_{k_0}(k_1,\ldots,k_d) \coloneqq (-1)^{k_0} \sum_{\substack{i_1 + \cdots + i_d = k_0 \\ i_1,\ldots,j_d\geq 0}} \tbinom{k_1 + i_1 - 1}{i_1} \cdots \tbinom{k_d + i_d - 1}{i_d} \reg_{\shuffle,0} \zeta(k_1 + i_1, \ldots, k_d + i_d) \,,
	\]
	We must retain \( \reg_{\shuffle,0} \) on the right-hand side, in the case \( k_d = 1 \).

	\subsection*{Index shuffle}  Finally, it is convenient to introduce the so-called \emph{index shuffle} \( \indsh \).  This is another product on \( \mathfrak{H}^1 \), defined recursively by
	\[
	\left\{ \begin{aligned}
	\quad z_{k_1} w_1 \indsh z_{k_2} w_2 &= z_{k_1} (w_1 \indsh z_{k_2} w_2) + z_{k_2} (z_{k_1} w_1 \indsh w_2)  \,; \\
	w \ast \one &= \one \ast w = w\,,
	\end{aligned} \right. 
	\]
	and extended by linearity.  Unlike the stuffle product, it does not generate any lower-depth terms as the third summand \( z_{k_1 + k_2} \) is missing.\medskip
	
	Note that \( \zeta \colon (\mathfrak{H}^1, \indsh) \to \mathbb{R} \) is explicitly \emph{not a morphism}.  For example 
	\begin{align*}
	 & \zeta(z_{k_1} \indsh z_{k_2}) - \zeta(z_{k_1})\zeta(z_{k_2})  = \zeta(k_1+k_2) \neq 0 \,.
	 \end{align*}
	 However, by the symmetric sum theorem \cite[\S2]{Hoffman92} the result \(
		\zeta(w_1 \indsh  w_2) \in \mathbb{Q}[\zeta(n) \mid n \geq 2] \,.
	\)  is a polynomial in Riemann zeta values.
	We are so not interested in the algebraic structure of properties of \( \indsh \); we simply use it as a convenient notation in which to write the main result.

	\section{Main theorem and proof}
	
	Here we give the precise statement of the main theorem, and prove it.  Let \( \vec{k} = (k_1,\ldots,k_d) \in \mathbb{Z}_{>0}^d \) be an index.  
	Recall the 1-variable multiple polylogarithm,
	\[
		\Li_{k_1,\ldots,k_d}(t) \coloneqq \sum_{0 < n_1 < n_2 < \cdots < n_d} \frac{t^{n_d}}{n_1^{k_1} n_2^{k_2} \cdots n_d^{k_d}} \,.
	\]
	Given another index \( \vec{\ell} \), write \( \vec{k}\vec{\ell} \) to denote the concatenation of indices, and define
	\begin{align*}
		\Psi(\vec{\ell} \mid x, y, z; t) &= \sum_{\substack{\vec{k} \in I(w, d, h)}} \Li_{\vec{k}\vec{l}}(t) x^{w-d-h} y^{d-h} z^h = \Li_{\vec{\ell}}(t) + \Li_{1,\vec{\ell}}(t) y + \Li_{2,\vec{\ell}}(t) z + \Li_{1,1,\vec{\ell}}(t) y^2 + \cdots \\
		\Psi_0(\vec{\ell} \mid x, y, z; t) &= \sum_{\substack{\vec{k} \in I_0(w, d, h)}} \Li_{\vec{k}\vec{l}}(t) x^w y^d z^{h-1} = \Li_{2,\vec{\ell}}(t)  + \Li_{3,\vec{\ell}}(t) x + \Li_{1,2,\vec{\ell}}(t) y + \cdots \,.
	\end{align*}
	In particular the series \( \Phi \), and \( \Phi_0 \) from Ohno-Zagier \cite{OhnoZagier01} are given by
	\begin{align*}
	\Phi(x, y, z; t) & \coloneqq \Psi(\emptyset \mid x, y, z; t) \,, \quad 
	\Phi_0(x, y, z; t) \coloneqq \Psi_0(\emptyset \mid x, y, z; t) \,.
	\end{align*}
	For \( \vec{\ell} = (\ell_1,\ldots,\ell_r) \) admissible, the main result is an expression for \( \Psi(\vec{\ell} \mid x, y, z; 1) \).
	\begin{Thm}\label{thm:phiell:gs}
		For any admissible index \( \vec{\ell} = (\ell_1,\ldots,\ell_r) \), with \( \ell_r > 1 \), the following generating series expression holds
		\[
		 \Phi(\vec{\ell} \mid x, y, z; 1) = \sum_{a,b,c=0}^\infty \reg_{\shuffle,0}
		 \zeta((z_1^a \indsh z_2^b) ( z_{\ell_1} - x z_{1+\ell_1}) (z_1^c \ast z_{\ell_2 }\cdots z_{\ell_r} ) )
		   (y-x)^a (z - x y)^b (-x)^c \,,
		\]
		where \( \indsh \) is the index shuffle, and \( \ast \) is the stuffle-product.
	\end{Thm}
It could be interesting to further extend this to allow regularisation and non-admissible \( \vec{\ell} \).  First we give some useful corollaries, the proof of \autoref{thm:phiell:gs} is presented afterwards.

\begin{Cor}\label{cor:wtsum:ell}
	Fix the weight \( w \), depth \( d \) and height-parameter \( \eta \).  For a fixed string \( \vec{\ell} = (\ell_1,\ldots,\ell_r) \), \( \ell_r \geq 2 \), the following weighted sum evaluation holds
	\[
	\sum_{\vec{k} \in I(w, d)} \binom{\wwt(\vec{k})}{\eta} \zeta(\vec{k}\vec{\ell}) 
	= \begin{aligned}[t]
	& \sum_{a + c = w - 2 \eta} 
	(-1)^{w + d + \eta} \binom{a}{d-\eta} \reg_{\shuffle,0}\zeta( (z_1^a \mathbin{\indsh} z_2^\eta) z_{\ell_1} (z_1^c \ast z_{\ell_2} \cdots z_{\ell_r}) ) \\
	& { + }\sum_{a + c = w -1 - 2 \eta} 
	(-1)^{w + d + \eta} \binom{a}{d-\eta} \reg_{\shuffle,0} \zeta( (z_1^a \mathbin{\indsh} z_2^\eta) z_{1+\ell_1} (z_1^c \ast z_{\ell_2} \cdots z_{\ell_r}) )
	\,.
	\end{aligned}
	\]
	
	\begin{proof}
		This follows directly from \autoref{thm:phiell:gs}, by extracting the coefficient of \( x^w y^d z^{\eta} \) in
		\[
			\phi(\vec{\ell} \mid x, x y, x^2 y z; 1) \,. \qedhere
		\]
	\end{proof}
\end{Cor}

\begin{Cor}
	The sum 
\[
\sum_{\vec{k} \in I(w, d)} \binom{\wwt(\vec{k})}{\eta} \zeta(\vec{k}\vec{\ell}) 
\]
is expressible via MZV's of depth \( \leq \eta + \wwt(\vec{\ell}) \).
	
	\begin{proof}
	Under duality, weight and depth pair \( (w,d) \) transforms to \( (w, w-d) \).  Note also, every term in \( z_{a_1} \cdots z_{a_r} \ast z_{b_1} \cdots z_{b_s} \) has depth \( \geq \max(r,s) \) as \( a_i \)'s never combine with themselves, nor \( b_j \)'s with themselves.  So after computing the index shuffle \( \indsh \) and stuffle \(\ast \) products in \autoref{cor:wtsum:ell}, we have MZV's of depth at least \( a + \eta + 1 + \max(c, r-1)  \geq a + \eta + c + 1 \), and weight \( a + 2 \eta + c + \wwt(\vec{\ell}) \), resp. \( a + 2 \eta + c + 1 + \wwt(\vec{\ell}) \).  Under duality, we obtain MZV's with depth \( \leq (a + 2 \eta + c + 1 + \wwt(\vec{\ell})) - (a + \eta + c + 1) = \eta + \wwt(\vec{\ell}) \), as claimed.
\end{proof}
\end{Cor}

\begin{proof}[Proof of \autoref{thm:phiell:gs}]
Recall the differential property of \( \Li_{k_1,\ldots,k_n}(t) \), 
\[
\frac{\dd}{\dd t} \Li_{k_1,\ldots,k_{n-1},k_n}(t) = \begin{cases}
(1-t)^{-1} \Li_{k_1,\ldots,k_{n-1}}(t) & \text{$k_n = 1$}, \\
t^{-1} \Li_{k_1,\ldots,k_{n-1},k_n-1}(t) & \text{$k_n \geq 1$} \,.
\end{cases}
\]
This is also valid for \( n = 1 \), with the convention that \( \Li_\emptyset(t) = 1 \).  Write \( D_1 = (1-t) \frac{\dd}{\dd t} \), and \( D_0 = t \frac{\dd}{\dd t} \), and take 
\[
 \mathcal{L} \coloneqq (D_0^{\ell_1-1}) (D_1 D_0^{\ell_2-1})  \cdots (D_1 D_0^{\ell_n-1}) \,; \]
 note the lack of \( D_1 \) in the first bracket.  Then we have
\begin{align*}\label{eqn:Lpsi}
\mathcal{L} \Psi(\vec{\ell} \mid x, y, z; t) = \sum_{\substack{\vec{k} \in I(w, d, h)}} \Li_{\vec{k},1}(t) x^{w-d-h} y^{d-h} z^h 
& = \frac{1}{y} \Big( \Phi(x,y,z;t) -1 - z \Phi_0(x,y,z;t) \Big)
\end{align*}

Ohno and Zagier \cite{OhnoZagier01} note that the following system of differential equations hold
\[
	\frac{\dd \Phi_0}{\dd t} = \frac{1}{y t} ( \Phi - 1 - z \Phi_0) + \frac{x}{t} \Phi_0 \,, \quad
	\frac{\dd}{\dd t} \big( \Phi - z \Phi_0 \big) = \frac{y}{1-t} \Phi \,.
\]
By eliminating \( \Phi \), they establish that \( \Phi_0 \) satisfies the second order differential equation
\[
	t(1-t) \frac{\dd^2 \Phi_0}{\dd t^2} + \big( (1-x)(1-y) - y t \big) \frac{\dd \Phi_0}{\dd t} + (x y - z) \Phi_0 = 1 \,.
\]
Finally they give a hypergeometric expression for \( \Phi_0 \) since it is the unique solution of the differential equation, which vanishes at \( t = 0 \).  
Recall the \( {}_2F_1 \) hypergeometric function is defined by
\[
\pFq{2}{1}{a,b}{c}{t} \coloneqq \sum_{m=0}^\infty \frac{\poch{a}{m} \poch{b}{m}}{\poch{c}{m}} \frac{t^m}{m!} \,, 
\]
where \( \poch{a}{m} \coloneqq a(a+1) \cdots(a+m-1)  \) is the Pochhammer symbol.  In particular Ohno and Zagier \cite{OhnoZagier01} find
\begin{equation}\label{eqn:phi0_solution}
\Phi_0(x,y,z;t) = \frac{1}{x y - z} \Big( 1 - \pFq{2}{1}{\alpha-x,\beta-x}{1-x}{t} \Big) \,,
\end{equation}
for \( \alpha + \beta = x+ y, \alpha \beta = z \).  
Hence we have
\begin{align*}
	\mathcal{L} \Psi(\vec{\ell} \mid x, y, z; t)
	& = \frac{1}{y} \Big( \Phi(x,y,z;t) -1 - z \Phi_0(x,y,z;t) \Big)  \\
	& =  \Big(t  \frac{\dd}{\dd t} - x \Big) \frac{1}{x y - z} \bigg( 1 - \pFq{2}{1}{\alpha-x,\beta-x}{1-x}{t} \bigg) \\
	& =   \frac{-1}{x y - z}  \sum_{m=1}^\infty \frac{\poch{\alpha-x}{m} \poch{\beta-x}{m}}{\poch{1-x}{m}} \Big( 1  - \frac{x}{m} \Big) \frac{t^m}{(m-1)!} \,.
\end{align*}
Note that
\begin{align*}
	\frac{-1}{x y - z} \frac{\poch{\alpha-x}{m} \poch{\beta-x}{m}}{\poch{1-x}{m}} \frac{1}{(m-1)!} 
	& = \frac{\prod_{p < m} \big( 1 + \frac{\alpha-x}{p} \big) \big( 1 + \frac{\beta-x}{p} \big)}{\prod_{q \leq m} \big( 1 - \frac{x}{q} \big) } \frac{1}{m} \\
	& = \frac{\prod_{p < m} \big( 1 + \frac{y-x}{p} + \frac{z-xy}{p^2} \big)}{\prod_{q \leq m} \big( 1 - \frac{x}{q} \big) } \frac{1}{m} \,,
\end{align*}
so
\[
\mathcal{L} \Psi(\vec{\ell} \mid x, y, z; t)
= \sum_{m=1}^\infty \frac{\prod_{p < m} \big( 1 + \frac{y-x}{p} + \frac{z-xy}{p^2} \big)}{\prod_{q \leq m} \big( 1 - \frac{x}{q} \big) } \Big( 1  - \frac{x}{m} \Big) \frac{t^m}{m} \,.
\]
Note that \( \Psi(\vec{\ell} \mid x, y, z; t) \) and all intermediate derivatives through to \( \mathcal{L}\Psi(\vec{\ell} \mid x, y, z; t) \) vanish at \( t = 0 \).  So to construct \( \Psi(\vec{\ell} \mid x, y, z; t) \) from the above expression, we apply the following operator
\[
	\mathcal{L}^{-1} \coloneqq (I_0^{\ell_n-1} I_1) \cdots (I_0^{\ell_2-1}I_1) (I_0^{\ell_1-1}) \,.
\]
where
\[
	I_0 \coloneqq \int_0^t \frac{\dd}{\dd t} \,, \quad \text{and} \quad  I_1 \coloneqq \int_0^t \frac{\dd t}{1-t} \,.
\]
By direct calculation
\begin{align*}
	I_0^{\ell_1-1} t^m &= \frac{t^m}{m^{\ell_1-1}} \,, \quad
	I_0^{\ell_j-1} I_1 t^m = 
	 \sum_{m < i} \frac{t^i}{i^{\ell_j}} \,.
\end{align*}
So applying \( \mathcal{L}^{-1} \) and taking \( t = 1 \) gives
\[
	\Psi(\ell \mid x, y, z; 1) = \sum_{m=1}^\infty \frac{\prod_{p < m} \big( 1 + \frac{y-x}{p} + \frac{z-xy}{p^2} \big)}{\prod_{q \leq m} \big( 1 - \frac{x}{q} \big) } \Big( 1  - \frac{x}{m} \Big) \sum_{\substack{m < i_2 < i_3 < \cdots \\ \cdots < i_{r-1} < i_r}} \frac{1}{m^{\ell_1} i_2^{\ell_2} \cdots i_r^{\ell_r}} \,.
\]
To extract a generating series expression form this, we first consider the truncated version
\[
	S_N \coloneqq  \sum_{m=1}^\infty \prod_{p < m} \Big( 1 + \frac{y-x}{p} + \frac{z-xy}{p^2} \Big) \cdot \frac{\prod_{m < q \leq N} \big( 1 - \frac{x}{q} \big) }{\prod_{q \leq N} \big( 1 - \frac{x}{q} \big) }  \Big( 1  - \frac{x}{m} \Big)  \sum_{\substack{m < i_2 < i_3 < \cdots \\ \cdots < i_{r-1} < i_r \leq N}} \frac{1}{m^{\ell_1} i_2^{\ell_2} \cdots i_r^{\ell_r}} \,,
\]
so that \( \Psi(\vec{\ell} \mid x, y, z; 1) = \lim_{N \to \infty} S_N \).  (Note the sum over \( m \) is actually restricted to \( m \leq N \).)  Recall the truncated multiple zeta (star) values are defined by restricting the summation indices in \eqref{eqn:def:mzv} (resp. \eqref{eqn:def:mzsv}) to a finite range, 
\begin{align*}
	\zeta_{\leq N}(\vec{k}) = \zeta_{\leq N}(k_1, k_2, \ldots, k_d) \coloneqq \sum_{0 < n_1 < n_2 < \cdots < n_d \leq N } \frac{1}{n_1^{k_1} n_2^{k_2} \cdots n_d^{k_d}} \,, \\
	\zeta_{\leq N}^\star(\vec{k}) = \zeta_{\leq N}^\star(k_1, k_2, \ldots, k_d) \coloneqq \sum_{0 < n_1 \leq n_2 \leq \cdots \leq n_d \leq N} \frac{1}{n_1^{k_1} n_2^{k_2} \cdots n_d^{k_d}} \,.
\end{align*}
(To avoid confusion, we have written \( \zeta_{\leq N} \) to denote the truncated MZV, as \( \reg_{\shuffle,0} \zeta_k \) denotes the shuffle-regularised MZV with \( k \) leading \( 0 \)'s in the integral representation.)  We note
\[
	\prod_{q \leq N} \Big( 1 - \frac{x}{q} \Big)^{-1} = \sum_{i=0}^\infty \zeta_{\leq N}^\star(\{1\}^i) x^i \,,
\]
compare for example \cite[Eqn. (44)]{zagier2232}.  Hence
\begin{align*}
	S_N = {} & \sum_{i=0}^\infty \zeta_{\leq N}^\star(\{1\}^i) x^i \\[-1ex]
	& \quad \cdot  \sum_{m=1}^\infty \prod_{p < m} \Big( 1 + \frac{y-x}{p} + \frac{z-xy}{p^2} \Big) \cdot \prod_{m < q \leq N} \Big( 1 - \frac{x}{q} \Big)   \sum_{\substack{m < i_2 < i_3 < \cdots \\ \mathclap{\cdots < i_{r-1} < i_r \leq N}}}  \Big( \frac{1}{m^{\ell_1}}  - \frac{x}{m^{\ell_1+1}} \Big) \frac{1}{i_2^{\ell_2} \cdots i_r^{\ell_r}} \,.
\end{align*}
Extracting the coefficient of \( (y-x)^a (z-xy)^b (-x)^c \) in the summation indexed by \( m \) gives
\begin{align*}
S_N = {} & \sum_{i=0}^\infty \zeta_{\leq N}^\star(\{1\}^i) x^i \cdot  \sum_{a,b,c=1}^\infty \zeta_{\leq N}((z_1^a \indsh z_2^b) (z_{\ell_1} - x z_{\ell_1+1}) (z_1^c \ast z_{\ell_2} \cdots z_{\ell_r})) (y-x)^a (z - xy)^b (-x)^c  \,.
\end{align*}
{\bf Explanation:} Compare again \cite[Eqn. (11)]{zagier2232}. For \( (y-x)^a (z-xy)^b \) we select either \( \frac{1}{p} \) or \( \frac{1}{p^2} \) so that there are \( a \)-many \( \frac{1}{p} \)'s (giving index \( z_1 \)) and \( b \)-many \( \frac{1}{p^2} \)'s (giving index \( z_2 \)) in total.  These can occur in any order, but summation indices are not equal, giving the index shuffle \( z_1^a \indsh z_2^b \).  For \( (-x)^c \), we have \( c \)-many terms \( \frac{1}{q} \) (giving index \( 1 \)), which must be interleaved with \( \frac{1}{i_2^{\ell_2}} \cdots \frac{1}{i_r^{\ell_r}} \).  The summation indices for \( z_1^c \) and \( z_{\ell_2} \cdots z_{\ell_r} \) are interleaved with equality allowed, giving the stuffle product \( z_1^c \ast z_{\ell_2} \cdots z_{\ell_r}\vphantom{\frac{1}{p}} \).  The index \( z_{\ell_1} \) (resp. \( z_{\ell_1+1} \) with additional coefficient \( -x\vphantom{\frac{1}{p}} \)) occurs between these two different products. As the largest summation index \( i_r \) is bounded above by \( N \), the result is the truncated MZV \( \zeta_{\leq N} \).\medskip

	Then expanding out the product of the \( \zeta_{\leq N}^\star\) generating series and the \( \zeta_N \) generating series gives
\begin{align*}
S_N = \sum_{a,b,c=1}^\infty \bigg( \sum_{i=0}^c \begin{aligned}[t] \zeta_{\leq N}((z_1^a \indsh z_2^b) (z_{\ell_1} - x z_{\ell_1+1}) (z_1^{c-i} \ast z_{\ell_2} \cdots z_{\ell_r}))  \cdot (-1)^i \zeta_{\leq N}^\star(\{1\}^i) \bigg) & \\ {} \cdot  (y-x)^a (z - xy)^b & (-x)^c \,.\end{aligned}
\end{align*}
The claim now is that
\begin{align}
	& \lim_{N\to\infty} \sum_{i=0}^c \zeta_{\leq N}((z_1^a \indsh z_2^b) (z_{\ell_1} - x z_{\ell_1+1}) (z_1^{c-i} \ast z_{\ell_2} \cdots z_{\ell_r}))  \cdot (-1)^i \zeta_{\leq N}^\star(\{1\}^i) \label{eqn:limit} \\
	& = \reg_{\shuffle,0} \zeta((z_1^a \indsh z_2^b) (z_{\ell_1} - x z_{\ell_1+1}) (z_1^{c-i} \ast z_{\ell_2} \cdots z_{\ell_r})) \notag \,.
\end{align}
The following lemma shows that this claim holds, by linearity.

\begin{Lem}\label{lem:limit}
	Let \( \vec{a} \) and \( \vec{b} \) be indices.  If the concatenation \( \vec{a}\vec{b}\) is admissible, then
	\[
		\lim_{N\to\infty} \sum_{i=0}^c \zeta_{\leq N}(z_{\vec{a}} (z_1^{c-i} \ast z_{\vec{b}}))  \cdot (-1)^i \zeta_{\leq N}^\star(\{1\}^i) = \reg_{\shuffle,0} \zeta(z_{\vec{a}} (z_1^{c-i} \ast z_{\vec{b}})) \,.
	\]
	Here for an index \( \vec{m} = (m_1,\ldots,m_d) \), the symbol \( z_{\vec{m}} \) denotes \( z_{m_1} \cdots z_{m_d} \).
\end{Lem}

Applying \autoref{lem:limit} to the limit in \autoref{eqn:limit}: after expanding linearly, every term in the limit is of the form \[ 
\sum_{i=1}^c \zeta_{\leq N}\big(z_{k_1} \cdots z_{k_s} z_{\ell_1 (+1)} \, (z_1^{c-i} \ast z_{\ell_2} \cdots z_{\ell_r})\big) \cdot (-1)^i \zeta_{\leq N}(\{ 1\}^{i}) \,,
\] with \( k_i \in \{ 1, 2 \} \).  In particular, for \( \vec{a} = (k_1, \ldots, k_s, \ell_1) \) (resp. \( \vec{a} = (k_1, \ldots, k_s, \ell_1+1) \)), and \( \vec{b} = (\ell_2, \ldots, \ell_r) \), the either \( r > 1 \) and \( \vec{a}\vec{b} \) ends in \( \ell_r > 1 \) by hypothesis of the \autoref{thm:phiell:gs}.  Otherwise \( r = 1 \) so \( \vec{b} = \emptyset \), and \( \vec{a}\vec{b} \) ends in \( \ell_1 > 1 \) (resp. \( \ell_1 + 1 > 2 \)) since \( \vec{\ell} = (\ell_1) \) is still admissible by hypothesis.\medskip

The limit of each individual term in \autoref{eqn:limit} can hence be computed by \autoref{lem:limit}.  After recombining with linearity in reverse, the we obtain the equality claimed in \autoref{lem:limit}.  With this \autoref{thm:phiell:gs} is proven.
\end{proof}

It remains to prove \autoref{lem:limit}.  We first treat the case \( \vec{b} = 0 \), with the following two stand alone lemmas.

\begin{Lem}\label{lem:trtost}
	For \( \vec{a} = (a_1,\ldots,a_d) \) admissible, the following limit holds, independent of the choice of regularisation parameter \( \reg_{\ast,T} \zeta(1) = T \):
	\[
		\lim_{N\to\infty} \sum_{i=0}^c \zeta_{\leq N}(\vec{a}, \{1\}^{c-i}) \cdot (-1)^i \zeta_{\leq N}^\star(\{1\}^i) = 
		\sum_{i=0}^c \reg_{\ast,T} \big( \zeta(\vec{a}, \{1\}^{c-i}) \cdot (-1)^i \zeta^\star(\{1\}^i) \big) \,.
	\]
	
	\begin{proof}
		We first establish a generating series regularisation identity, to see that the sum does not depend on \( \zeta_{\leq N}(1) \).  From \cite[Corollary 5, proof]{IKZ06} (note the reversed convention), we have the generating series identity
		\[
			z_{\vec{a}} \frac{1}{1 - z_1X} = \reg_{\ast,0}\Big(z_{\vec{a}} \frac{1}{1 - z_1X}\Big) \ast \exp_\ast(z_1 X) \,.
		\]
		Recall the ``starification''-map \( \Sigma \) from \cite[\S3.1]{HoffmanIhara17}, so that \( \zeta^\star = \zeta \circ \Sigma \) and \( \zeta_{\leq N}^\star = \zeta_{\leq N} \circ \Sigma \).  Applying \( \Sigma \) to both sides of \cite[Eqn. (34)]{HoffmanIhara17}, in the case \( k=1 \) gives
		\[
			\sum_{n=0}^\infty \Sigma(z_1^n) X^n = \exp_\ast\Big(\sum_{i=1}^\infty \frac{z_i X^i}{i} \Big) \,,
		\]
		as \( \Sigma \exp_\star = \exp_\ast \Sigma \) (see the discussion around Eqn. (25) in \cite{HoffmanIhara17}).  We therefore see that
		\begin{align*}
			z_{\vec{a}} \frac{1}{1 - z_1X} \ast \sum_{n=0}^\infty \Sigma(z_1^n) (-X)^n 
			& = \reg_{\ast,0}\Big(z_{\vec{a}} \frac{1}{1 - z_1X}\Big) \ast \exp_\ast(z_1 X) \ast \exp_\ast\Big(\sum_{i=1}^\infty \frac{z_i (-X)^i}{i} \Big) \\
			& = \reg_{\ast,0}\Big(z_{\vec{a}} \frac{1}{1 - z_1X}\Big) \ast \exp_\ast\Big(\sum_{i=2}^\infty \frac{z_i (-X)^i}{i} \Big) \,,
		\end{align*}
		 so the dependence on \( z_1 \) drops out.  In particular
		 \[
		 z_{\vec{a}} \frac{1}{1 - z_1X} \ast \sum_{n=0}^\infty \Sigma(z_1^n) (-X)^n 
		 = \reg_{\ast, T} \Big( z_{\vec{a}} \frac{1}{1 - z_1X} \ast \sum_{n=0}^\infty \Sigma(z_1^n) (-X)^n  \Big) \,.
		 \]
		 Applying \( \lim_{N\to\infty} \zeta_{\leq N} \) to the coefficient of \( X^c \) in this identity gives the claim, since \( \lim_{N\to\infty} \zeta_{\leq N} = \zeta \) on convergent words.
	\end{proof}
\end{Lem}

\begin{Lem}\label{lem:sttosh}
	For \( \vec{a} = (a_1,\ldots,a_d) \) admissible, the following relation between the stuffle-regularisation and shuffle-regularisation holds,
	\[
		\sum_{i=0}^c \reg_{\ast,T} \big( \zeta(\vec{a}, \{1\}^{c-i}) \cdot (-1)^i \zeta^\star(\{1\}^i) \big) = \reg_{\shuffle,0} \zeta(\vec{a}, \{1\}^c) \,.
	\]
	
	\begin{proof}
		It was shown in \cite[Lemma 2.9]{charltont2212} (and is likely well-known elsewhere) that
		\[
			\reg_{\ast,T} \zeta(\vec{a}, \{1\}^c) = \sum_{i=0}^c \reg_{\shuffle,0} \zeta(\vec{a} ,\{1\}^{c-i}) \cdot \reg_{\ast,T} \zeta(\{1\}^i) \,.
		\]
		Forming the generating series \( \sum_{c=0}^\infty \bullet X^c \), and multiplying both sides by (see \cite[\S6.1]{HoffmanIhara17})
		\[
			\sum_{i=0}^\infty \reg_{\ast,T} \zeta^\star(\{1\}^i) (-X)^j = \exp\bigg( \sum_{i=1}^\infty  \frac{\reg_{\ast,T} \zeta(i) (-X)^i}{i} \bigg) = \bigg( \sum_{i=0}^\infty \reg_{\ast,T} \zeta(\{1\}^i) X^i \bigg)^{-1} \,,
		\]
		gives the result.
	\end{proof}
\end{Lem}
	
	\begin{proof}[Proof of \autoref{lem:limit}]
		This is by induction on the parameter \( c + \ddp(\vec{b}) \).  When \( c = 0 \), statement reduces to
		\[
			\lim_{N\to\infty} \zeta_{\leq N}(z_{\vec{a}} z_{\vec{b}}) = \zeta(z_{\vec{a}} z_{\vec{b}}) = \reg_{\shuffle,0} \zeta(z_{\vec{a}} z_{\vec{b}}) \,,
		\]
		which holds as \( \vec{a}\vec{b} \) is admissible.  Likewise, if \( \ddp(\vec{b}) = 0 \), meaning \( \vec{b} = \emptyset \), then the statement reduces to
		\[
		\lim_{N\to\infty} \sum_{i=0}^c \zeta_{\leq N}(z_{\vec{a}} z_1^{c-i})  \cdot (-1)^i \zeta_{\leq N}^\star(\{1\}^i) = \reg_{\shuffle,0} \zeta(z_{\vec{a}} z_1^{c-i}) \,,
		\]
		which holds via \autoref{lem:trtost} and \autoref{lem:sttosh}, as \( \vec{a}\vec{b} = \vec{a}\) is admissible. This deals with the base cases.  \medskip
		
		Now suppose the claim holds whenever \( c'  + \ddp(\vec{b}') < c + \ddp(\vec{b}) \).  Since \( \vec{b} = (b_1,\ldots,b_d) \neq \emptyset \), write \( \vec{b}' = (b_2,\ldots,b_d) \).  Then for \( c-i > 0 \), we have by the recursive definition of the \( \ast \) product, 
				\[
			z_{\vec{a}} (z_1^{c-i} \ast z_{\vec{b}}) = z_{\vec{a}} z_{b_1} (z_1^{c-i} \ast z_{\vec{b}'})  + z_{\vec{a}} z_1 (z_1^{c-i-1} \ast z_{\vec{b}}) + z_{\vec{a}} z_{b_1+1} (z_1^{c-i-1} \ast z_{\vec{b}'}) \,.
 		\]
 		We can then write
		\begin{align*}
			& \sum_{i=0}^c \zeta_{\leq N}(z_{\vec{a}} (z_1^{c-i} \ast z_{\vec{b}}))  \cdot (-1)^i \zeta_{\leq N}^\star(\{1\}^i) \\
			& = \zeta_{\leq N}(z_{\vec{a}} (z_1^{0} \ast z_{\vec{b}}))  \cdot (-1)^c \zeta_{\leq N}^\star(\{1\}^c)+ \sum_{i=0}^{c-1} \zeta_{\leq N}(z_{\vec{a}} (z_1^{c-i} \ast z_{\vec{b}}))  \cdot (-1)^i \zeta_{\leq N}^\star(\{1\}^i) \\
			& = \begin{aligned}[t] 
				\zeta_{\leq N}(z_{\vec{a}} (z_1^{0} \ast z_{\vec{b}})) & \cdot (-1)^c \zeta_{\leq N}^\star(\{1\}^c)+ \sum_{i=0}^{c-1} \Big ( \zeta_{\leq N}(z_{\vec{a}} z_{b_1} (z_1^{c-i} \ast z_{\vec{b}'}))  \\
				& + \zeta_{\leq N}(z_{\vec{a}} z_1 (z_1^{c-i-1} \ast z_{\vec{b}})) +  \zeta_{\leq N}(z_{\vec{a}} z_{b_1+1} (z_1^{c-i-1} \ast z_{\vec{b}'})) \Big) \cdot (-1)^i \zeta_{\leq N}^\star(\{1\}^i)  \\
			\end{aligned} 
		\end{align*}
		Since \( z_{\vec{a}} (z_1^{0} \ast z_{\vec{b}})) = z_{\vec{a}} z_{b_1} (z_1^0 \ast z_{\vec{b}'}) \), we can repackage the first line into a single sum, giving
		\begin{align*}
		& = \begin{aligned}[t] 
		& \sum_{i=0}^{c} \Big ( \zeta_{\leq N}(z_{\vec{a}} z_{b_1} (z_1^{c-i} \ast z_{\vec{b}'})) \cdot (-1)^i \zeta_{\leq N}^\star(\{1\}^i) \\
		& + \sum_{i=0}^{c-1} \Big( \zeta_{\leq N}(z_{\vec{a}} z_1 (z_1^{c-i-1} \ast z_{\vec{b}})) +  \zeta_{\leq N}(z_{\vec{a}} z_{b_1+1} (z_1^{c-i-1} \ast z_{\vec{b}'})) \Big) \cdot (-1)^i \zeta_{\leq N}^\star(\{1\}^i)  \\
		\end{aligned} 
		\end{align*}
		The first sum has induction parameter \( c + \ddp(\vec{b}') \) and  \( (\vec{a},b_1)\vec{b}' = \vec{a}\vec{b} \) is admissible.  The second has induction parameter \( c-1 + \ddp(\vec{b}) \) and \( (\vec{a},1)\vec{b} \) is admissible.   The third has induction parameter \( c-1 + \ddp(\vec{b}') \), and \( (\vec{a},b_1+1)\vec{b}' \) is admissible as it either ends in \( b_1 + 1 > 1 \) (when \( \vec{b}' = \emptyset \)) or in \( b_d > 1 \).  The induction hypothesis therefore allows us to compute the limit, giving
		\begin{align*}
		 \xrightarrow{N\to\infty} {} & \reg_{\shuffle,0} \big( \zeta(z_{\vec{a}} z_{b_1} (z_1^{c} \ast z_{\vec{b}'})) + \zeta(z_{\vec{a}} z_1 (z_1^{c-1} \ast z_{\vec{b}})) + \zeta(z_{\vec{a}} z_{b_1+1} (z_1^{c-1} \ast z_{\vec{b}'})) \big) \\
		 {} = {} & \reg_{\shuffle,0} \zeta(z_{\vec{a}} ( z_1^c \ast z_{\vec{b}})) \,.
		\end{align*}
		The final line is obtained by applying the recursive definition of the \( \ast \)-product.  This has established the inductive step, so the proof is complete.
\end{proof}
	
	\section{Applications to MZV evaluations}
	We end with some applications to the evaluations of multiple zeta values or their generating series.
	
	\subsection{\texorpdfstring{Case \( \zeta^r(2,\{1\}^n,2) \)}
		{Case zeta\textasciicircum{}r(2,\{1\}\textasciicircum{}n,2)}}\label{sec:zh:2112}
	
	To evaluate \( \zeta^r(2, \{1\}^{m-4}, 2) \), we first note that
	\begin{align*}
		\zeta^r(2,\{1\}^{m-4},2) &\coloneqq \sum_{\substack{(k_1,\ldots,k_d) \in I_0(m,d) \\ k_1 \geq 2, d \geq 1}} r^{m-2-d} \zeta(k_1,\ldots,k_d) \\
		& = \sum_{\substack{(k_1,\ldots,k_d) \in I_0(m,d) \\ k_d = 2, d \geq 1}} r^{d-2} \zeta(k_1,\ldots,k_d) \,,
	\end{align*}
	using duality.  By \autoref{cor:wtsum:ell}, with \( \vec{\ell} = (2) \) and \( \eta = 0 \), we have, after applying duality that
	\begin{align*}
	& \sum_{\substack{(k_1,\ldots,k_d) \in I_0(w,d) \\ k_d = 2}} \zeta(k_1,\ldots,k_d) \\
	& = (-1)^{w+d-1} \reg_{\shuffle,0} \bigg\{ \sum_{i+j=w-2} \tbinom{i}{d-1} \zeta_j(i+2) + \sum_{i+j=w-3} \tbinom{i}{d-1} \zeta_j(1,i+2) \bigg\}
	\end{align*}
	Forming the generating series, we have
	\begin{align*}
	& \sum_{n=4}^\infty \zeta^r(2,\{1\}^{n-4},2) x^{n-1}  \\
	& = \sum_{w=4}^\infty \sum_{d=1}^{w-1} (-1)^{w+d-1} r^{d-2} \reg_{\shuffle,0} \Bigg\{ \sum_{i+j=w-2} \tbinom{i}{d-1} \zeta_j(i+2) + \sum_{i+j=w-3} \tbinom{i}{d-1} \zeta_j(1,i+2) \Bigg\} \\
	& = -\zeta(2)\frac{x}{r} - \zeta(3)x^2  + \sum_{i,j=0}^\infty \sum_{d=1}^\infty (-r)^{d-2} \tbinom{i}{d-1}  \reg_{\shuffle,0} \Bigg\{ 
	  \zeta_j(i+2) (-x)^{i+j+1}  +  \zeta_j(1,i+2) (-x)^{i+j+2} 
	\Bigg\} \,. 
	\end{align*}
	Here, the terms of weight 2 and 3 arise by extending the sum to include \( w = 2, 3 \), and removing the extra terms manually.  Now summing over \( d \), gives
	\begin{equation}\label{eqn:zr2112:gs}
	= -\frac{\zeta(2)x}{r} - \zeta(3)x^2  - \sum_{i,j=0}^\infty\frac{(1-r)^i}{r}  \reg_{\shuffle,0} \Bigg\{ 
	\zeta_j(i+2) (-x)^{i+j+1} + \zeta_j(1,i+2) (-x)^{i+j+2} \Bigg\} \,.
	\end{equation}
	Using the regularisation formulas
	\begin{align*}
		\reg_{\shuffle,0} \zeta_c(a) &= (-1)^c \sum_{i_1=c} \tbinom{i_1+a-1}{i_1} \reg_{\shuffle,0} \zeta(a+i_1)  \,, \\
		\reg_{\shuffle,0} \zeta_c(a,b) &= (-1)^c \sum_{i_1+i_2=c} \tbinom{i_1+a-1}{i_1} \tbinom{i_1+b-1}{i_2} \reg_{\shuffle,0} \zeta(a+i_1,b+i_2)  \,,
	\end{align*}
	we have by a direction calculation, the following generating series identities
	\begin{align*}
	\sum_{b=0}^\infty  \reg_{\shuffle,0} \zeta_c(b+2) z^{b+1} w^{c} &= \sum_{b=1}^\infty \zeta(b+1) \big\{ (z - w)^b - (- w)^b \big\} \,. \\
	\sum_{b,c=0}^\infty  \reg_{\shuffle,0} \zeta_c(a+1,b+2) y^{a} z^{b+1} w^{c} &= \sum_{a=0}^\infty \sum_{b=1}^\infty  \zeta(a+1,b+1) (y - w)^a \big\{ (z-w)^b -  (-w)^b \} \,,
	\end{align*}
	(since the \( b = 0 \) summand is identically 0).  Applying these to \autoref{eqn:zr2112:gs}, with \( y = 0, z = -(1-r)x, w = -x \), we find that
	\begin{align*}
	 \sum_{n=0}^\infty \zeta^r(2,\{1\}^{n-4},2) x^{n-1}  
	=  -\frac{\zeta(2)x}{r} - \zeta(3)x^2  & {} - \sum_{b=1}^\infty \frac{\zeta(b+1)}{(1-r)r} \big\{ (rx)^{b} - x^b \big\}
	\\
	& \qquad + \sum_{a=0}^\infty \sum_{b=1}^\infty  \frac{\zeta(a+1,b+1)}{(1-r)r} x^{a+1} \big\{ (rx)^b - x^b \big\} \,.
	\end{align*}
	Using the sum formula \( \sum_{i=1}^{n-2} \zeta(i, n-i) = \zeta(n) \), one can cancel the \( x^b \) terms in weight \( {\geq}3 \) in series above, giving
	\begin{align*}
	\sum_{n=0}^\infty \zeta^r(2,\{1\}^{n-4},2) x^{n-1}  
	=  \frac{\zeta(2)x}{1-r} - \zeta(3)x^2  & {} - \sum_{b=1}^\infty \frac{\zeta(b+1)}{(1-r)r} (rx)^{b}  + \sum_{a=0}^\infty \sum_{b=1}^\infty  \frac{\zeta(a+1,b+1)}{(1-r)r} x^{a+1} (rx)^b \,.
	\end{align*}
	Then extracting the coefficient of \( x^{n+3} \), leads to the following evaluation for \( \zeta^r(2, \{1\}^n, 2) \) in terms of double zeta values; in odd weight the parity theorem \cite[\S2.6]{GoncharovMultiple01}, \cite{PanzerParity16} reduces this to a polynomial in Riemann zeta values.
	\begin{Prop}[Evaluation of \( \zeta^r(2,\{1\}^n,2) \)]\label{prop:zr:2112}
		The following evaluation holds for all \( n \geq 0  \),
		\[
		\zeta^r(2,\{1\}^n,2) = \frac{1}{(1-r)r^2} \sum_{\substack{i+j = n+4 \\ i \geq 1,j \geq 2}} r^j \zeta(i,j) - \frac{1}{(1-r)r^2} \cdot r^{n+4} \zeta(n+4) \,.
		\]
		In the case \( n = 2m + 1\) odd, for \( m \geq 0 \), we have the following evaluation in terms of Riemann zeta values,
		\[
		\zeta^r(2,\{1\}^{2m+1},2) = \begin{aligned}[t] & \bigg( \frac{m+2}{r(1-r)} - \frac{1 - (1-r)^{2m+4}}{2r^2} - \frac{1 - r^{2m+4}}{2(1-r)^2} \bigg) \zeta(2m+5) \\
		& - \frac{1}{r(1-r)} \sum_{\substack
			{2s + k = 2m + 5 \\ s \geq 1, k \geq 2}} (1 - r^{2s-1})(1 - (1-r)^{k-1}) \zeta(2s)\zeta(k) \,.
		\end{aligned}
		\]
	\end{Prop}

	In the special case where \( r = \half \), we recall the weighted sum formula from \cite[Thm. 3]{OhnoZudilinZetaStars} (there with the opposite convention),
	\[
		\sum_{i=2}^{n-1} 2^{i} \zeta(n-i,i) = (n+1) \zeta(n) \,.
	\]
	Using that 
	\[
		\zeta(1,n-1) = \frac{n-1}{2} \zeta(n) + \frac{1}{2} \sum_{i=2}^{n-2} \zeta(i)\zeta(n-i) \,,
	\]
	via the generating series in \cite[Eqn. 10]{BBBkfold}, and otherwise rewriting \( \zeta(n-i,i) = \zeta(n-i)\zeta(i) - \zeta(i,n-i) - \zeta(n) \), for \( 2 \leq i \leq n-2 \), it follows that 
	\[
		\sum_{i=1}^{n-2} 2^{-(n-i)} \zeta(i,n-i)  = \frac{1}{4} \sum_{i=2}^{n-2} (1 + 2^{2-n} - 2^{2+i-n}) \zeta(i) \zeta(n-i) + \Big( 2^{1-n} + \frac{n-3}{4} \Big) \zeta(n)  \,.
	\]
	Substituting this into \autoref{prop:zr:2112}, we find that \( \zeta^\half(2, \{1\}^n, 2) \) always evaluates as a polynomial in Riemann zeta values.  In particular, we obtain the following.
	\begin{Prop}[Evaluation of \( \zeta^\half(2,\{1\}^{n},2) \)]\label{prop:zh:2112}
		The following evaluation hold for all \( n \geq 0 \),
		\begin{align*}
		\zeta^\half(2,\{1\}^n,2) = \sum_{\substack{i+j=n+4 \\ i,j \geq 2}} {-}2^{1-i-j}(2 - 2^i)(2 - 2^j) \zeta(i)\zeta(j) + (2^{-1-n} + 2n + 2) \zeta(n+4) \,.
		\end{align*}
	\end{Prop}
	
	\subsection{Hypergeometric generating series of double zeta values}
	
	Recall the generalised hypergeometric function
	\[
	\pFq{p}{q}{a_1,\ldots,a_p}{b_1,\ldots,b_q}{t} \coloneqq \sum_{m=0}^\infty \frac{\poch{a_1}{m} \cdots \poch{a_p}{m}}{\poch{b_1}{m} \cdots \poch{b_q}{m}} \frac{t^m}{m!} \,,
	\]
	where \( \poch{x}{m} = x (x+1) \cdots (x+m-1) \) is the Pochhammer symbol.  From \autoref{lem:limit}, we can obtain a hypergeometric expression for the generating series of double zeta values.  This could be used to study MZV evaluations via double zeta values, using automatic summation, creative telescoping and WZ methods as in \cite{AuCtel}.
	
		\begin{Prop}[Generating series of double zeta values]\label{prop:dzv:hyper}
		The following gives a closed form hypergeometric expression for the generating series double zeta values:
		\begin{align*}
		 \sum_{c,d=0}^\infty \zeta(c+1,d+2) x^{c} y^{d+1} 
		 = \frac{x-y}{x-1} \pFq{4}{3}{1,1,1,1-x+y}{2,2,2-x}{1} - \zeta(2)
		  - \frac{\psi(1-x) + \gamma }{x} \,,
		\end{align*}
		where \( \psi(x) \coloneqq \frac{\mathrm{d}}{\mathrm{d}x} \log\Gamma(x) \) is the digamma function.
		\end{Prop}
	
		\begin{proof}
			Since
			\[
				 \frac{\poch{1}{m-1} \poch{y-x}{m}}{\poch{1-x}{m} \cdot m} \frac{1}{m!} 
				= \frac{x-y}{x-1} \frac{\poch{1}{m-1}^3 \poch{1-x+y}{m-1}}{\poch{2}{m-1}^2 \poch{2-x}{m-1}} \frac{1}{(m-1)!} \,,
			\]
			we have that
			\[
				 \sum_{m=1}^\infty \frac{\poch{1}{m-1} \poch{y-x}{m}}{\poch{1-x}{m} \cdot} \frac{1}{m!}  = \frac{x-y}{x-1} \pFq{4}{3}{1,1,1,1-x+y}{2,2,2-x}{1} \,.
			\]
			On the other hand,
			\[
				\frac{\poch{1}{m-1} \poch{y-x}{m}}{\poch{1-x}{m} \cdot m} \frac{1}{m!} = 
				\frac{(y-x) \prod_{p < m} \big(1 + \frac{y-x}{p} \big)}{m^3 \prod_{q \leq m} \big( 1 - \frac{x}{q} \big)} \,.
			\]
			Using the same argument as in the proof of \autoref{thm:phiell:gs}, we have
			\begin{align*}
				& \sum_{m=1}^N \frac{(y-x) \prod_{p < m} \big(1 + \frac{y-x}{p} \big)}{m^3 \prod_{q \leq m} \big( 1 - \frac{x}{q} \big)} \\
				& = (y-x)  \prod_{q \leq N} \Big( 1 - \frac{x}{q} \Big)^{-1} \cdot 
				\sum_{m=1}^N \prod_{p < m} \bigg(1 + \frac{y-x}{p} \bigg) \frac{1}{m^3} \prod_{m < q \leq N} \Big( 1 - \frac{x}{q} \Big) \\
				& = \sum_{i=0}^\infty \zeta_{\leq N}^\star(\{1\}^i) x^i \cdot
				\sum_{a,b=0}^\infty \zeta_{\leq N}(\{1\}^a, 3, \{1\}^b) (y-x)^{a+1} (-x)^b \\
				& = \sum_{a,b=0}^\infty \bigg( \sum_{i=0}^b \zeta_{\leq N}(\{1\}^a, 3, \{1\}^{b-i}) \cdot (-1)^i \zeta_{\leq N}^\star(\{1\}^i) \bigg) (y-x)^{a+1} (-x)^b
			\end{align*}
			From \autoref{lem:limit} and then applying duality, we obtain that the limit as \( N \to \infty \) is
			\begin{align*}
			\xrightarrow{N\to\infty} {} & 
			 \sum_{a,b=0}^\infty \reg_{\shuffle,0} \zeta(\{1\}^a, 3, \{1\}^b)  (y-x)^{a+1} (-x)^b = 
			 \sum_{a,b=0}^\infty \reg_{\shuffle,0} \zeta_b(1, a+2) (y-x)^{a+1} (-x)^b \,,
			\end{align*}
			Using the regularisation formula
			\[
			\reg_{\shuffle,0} \zeta_b(1, a)
			= (-1)^b \sum_{i_1 + i_2 = b} \tbinom{a + 1 + i_2}{i_2} \reg_{\shuffle,0} \zeta(1 + i_1, a + i_2) 
			\]
			we find 
			\begin{align*}
			 \sum_{a,b=0}^\infty \reg_{\shuffle,0} \zeta_b(1, a+2) (y-x)^{a+1} (-x)^b
			& = \sum_{c,d=0}^\infty \zeta(c+1, d+2) x^{c} y^{d+1} - \sum_{c,d=0}^\infty \zeta(c+1, d+2) x^{c+d+1} \,.
			\end{align*}
			Using the sum formula \( \sum_{i=1}^{n-2} \zeta(i,n-i) = \zeta(n) \), gives
			\[
				\sum_{c,d=0}^\infty \zeta(c+1, d+2) x^{c+d+1} 
				= 
				\sum_{k=0}^\infty \zeta(k+3) x^{k+1}
				= 
				\sum_{k=0}^\infty \zeta(k+3) x^{k+1} = -\zeta(2) - \frac{\psi(1-x) + \gamma}{x} \,,
			\]
			where \( \psi(x) = \frac{\dd}{\dd x} \log{\Gamma(x)} \) (for the Taylor series, see for example \cite[\S5.7(i), Eqn. (5.7.4)]{NIST:DLMF}).  So the result follows.
		\end{proof}	

	\bibliographystyle{habbrv2} 
	\bibliography{bibliography.bib} 

\end{document}